\documentclass[12pt]{amsart}
\usepackage{amsmath,amssymb,amsthm,mathdots}
\usepackage{rotating}
\usepackage{hyperref}
\usepackage{latexsym}            % for the qed symbol
\usepackage{colordvi,graphicx}
\usepackage{hyperref}
\usepackage{bm} % bold math
\numberwithin{equation}{section}

\newtheorem{theorem}{Theorem}[section]

\newtheorem{lemma}[theorem]{Lemma}
\newtheorem{corollary}[theorem]{Corollary}

\theoremstyle{remark}

\newtheorem{example}[theorem]{Example}

\DeclareMathOperator{\codim}{codim}

\makeatletter
\def\Ddots{\mathinner{\mkern1mu\raise\p@
\vbox{\kern7\p@\hbox{.}}\mkern2mu
\raise4\p@\hbox{.}\mkern2mu\raise7\p@\hbox{.}\mkern1mu}}
\makeatother

\newcounter{FNC}[page]
\def\fauxfootnote#1{{\addtocounter{FNC}{2}\Magenta{$^\fnsymbol{FNC}$}%
     \let\thefootnote\relax\footnotetext{\Magenta{$^\fnsymbol{FNC}$#1}}}}

\DeclareMathOperator{\Gr}{Gr}
\DeclareMathOperator{\rank}{rank}
\DeclareMathOperator{\Span}{Span}

\newcommand{\be}{{\bf e}}
\newcommand{\bff}{{\bf f}}
\newcommand{\bbeta}{{\bm \beta}}

\newcommand{\C}{{\mathbb C}}
\newcommand{\calO}{{\mathcal O}}

\newcommand{\Fdot}{F_\bullet}
\newcommand{\Edot}{E_\bullet}

\newcommand{\I}{\raisebox{-1pt}{\includegraphics{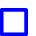}}}
\newcommand{\sI}{\includegraphics{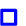}}
\newcommand{\III}{\raisebox{-1pt}{\includegraphics{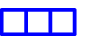}}}
\newcommand{\sIII}{\includegraphics{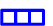}}

\newcommand{\defcolor}[1]{\Blue{#1}}
\newcommand{\demph}[1]{\defcolor{{\sl #1}}}
\title[Certifiable Numerical Computations in Schubert Calculus]{Certifiable Numerical
  Computations\\ in Schubert Calculus} 

\author{Jonathan D. Hauenstein, Nickolas Hein, Frank Sottile}
\address{Jonathan D. Hauenstein \\
         Department of Mathematics\\
         North Carolina State University\\
         Raleigh\\
         North Carolina \ 27695\\
         USA}
\email{hauenstein@ncsu.edu}
\urladdr{\url{http://www.math.ncsu.edu/~jdhauens/}}
\address{Nickolas Hein \\
         Department of Mathematics\\
         Texas A\&M University\\
         College Station\\
         Texas \ 77843\\
         USA}
\email{nhein@math.tamu.edu}
\urladdr{\url{http://www.math.tamu.edu/~nhein}}
\address{Frank Sottile \\
         Department of Mathematics\\
         Texas A\&M University\\
         College Station\\
         Texas \ 77843\\
         USA}
\email{sottile@math.tamu.edu}
\urladdr{\url{http://www.math.tamu.edu/~sottile/}}
\thanks{Research of Sottile and Hein supported in part by NSF grant DMS-0915211.
        Research of Hauenstein supported in part by AFOSR grant FA8650-13-1-7317 and NSF grant DMS-1262428.}
\PII{}
\copyrightinfo{}{}
\begin{document}
%%%%%%%%%%%%%%%%%%%%%%%%%%%%%%%%%%%%%%%%%%%%%%%%%%%%%%%%%%%%%%%%%%%%%%%%%%%%%%%%%%%%%%%%%%%%%
\begin{abstract}
 Traditional formulations of geometric problems from the Schubert calculus, either in
 Pl\"ucker coordinates or in local coordinates provided by Schubert cells, yield systems
 of polynomials that are typically far from complete intersections and (in local
 coordinates) typically of degree exceeding two.
 We present an alternative primal-dual formulation using parametrizations of Schubert
 cells in the dual Grassmannians in which intersections of Schubert varieties become
 complete intersections of bilinear equations.
 This formulation enables the numerical certification of problems in the Schubert
 calculus.
\end{abstract}
%%%%%%%%%%%%%%%%%%%%%%%%%%%%%%%%%%%%%%%%%%%%%%%%%%%%%%%%%%%%%%%%%%%%%%%%%%%%%%%%%%%%%%%%%%%%%
\maketitle
%%%%%%%%%%%%%%%%%%%%%%%%%%%%%%%%%%%%%%%%%%%%%%%%%%%%%%%%%%%%%%%%%%%%%%%%%%%%%%%%%%%%%%%%%%%%%
%
\section*{Introduction}
Numerical nonlinear algebra provides algorithms that certify numerically computed
solutions to a system of polynomial equations, provided that the system is square---the
number of equations is equal to the number of variables.
To use these algorithms for certifying results obtained through numerical computation in
algebraic geometry requires that we use equations which exhibit our varieties as complete
intersections.
While varieties are rarely global complete intersections, it suffices to have a
{\sl local} formulation in the folowing sense:
The variety has an open dense set which our equations exhibit as a complete intersection
in some affine space.
If it is zero-dimensional, then we require the variety to be a complete intersection in some
affine space.
Here, we use a primal-dual formulation of Schubert varieties to formulate all
problems in Schubert calculus on a Grassmannian as complete intersections, and indicate
how this extends to all classical flag manifolds.

The Schubert calculus of enumerative geometry has come to mean all problems which involve
determining the linear subspaces of a vector space that have specified positions with
respect to other fixed, but general, linear subspaces.
It originated in work of Schubert~\cite{Sch1879} and others to solve geometric problems
and was systemized in the 1880's~\cite{Sch1886c,Sch1886a,Sch1886b}.
Most work has been concerned with understanding the numbers of solutions to problems in
the Schubert calculus, particularly finding~\cite{LR1934},
proving~\cite{Schutzenberger_jeu_de_taquin,Thomas}, and generalizing the
Littlewood-Richardson rule.
As a rich and well-understood class of geometric problems, the Schubert calculus is a
laboratory for the systematic study of new phenomena in enumerative
geometry~\cite{So_FRSC}.
This study requires that Schubert problems be modeled and solved on a computer.

Symbolic methods, based on Gr\"obner bases and elimination theory, are well-understood and
quite general.
They are readily applied to solving Schubert problems---their use was
central to uncovering evidence for the Shapiro Conjecture~\cite{So00b} as well as
formulating its generalizations~\cite{secant,monotone_MEGA,RSSS}.
An advantage of symbolic methods is that they are exact---a successful computation is a
proof that the outcome is as claimed.
This exactness is also a limitation, particularly for Gr\"obner bases.
The output of a Gr\"obner basis computation contains essentially all the information of the
object computed, and this is one reason for the abysmal complexity of Gr\"obner
bases~\cite{MM82}, including that of zero-dimensional ideals~\cite{HL10}.

Besides fundamental complexity, another limitation on Gr\"obner bases is that they do not
appear to be parallelizable.
This matters since the predictions of Moore's Law are now fulfilled through increased
processor parallelism, and not by increased processor speed.
Numerical methods based upon homotopy continuation~\cite{SW05} offer an attractive
parallelizable alternative.
A drawback to numerical methods is that they
do not intrinsically come with a proof that their output is as claimed, and for the
Schubert calculus, standard homotopies perform poorly since the  
upper bounds on the number of solutions on which they are based (total degree or mixed
volume), drastically overestimate the true number of solutions.  
The Pieri homotopy~\cite{HSS98} and
Littlewood-Richardson homotopy~\cite{svv} are optimal homotopy methods
which are limited to Schubert calculus on the Grassmannian.
A numerical approximation to a solution of a system of polynomial equations may be refined
using Newton's method and we call each such refinement a Newton iteration.
Smale analyzed the convergence of repeated Newton iterations, when the system is
square~\cite{S86}.
The name, $\alpha$-theory, for this study refers to a constant $\alpha$ which depends upon
the approximate solution $x_0$ and system $f$ of polynomials~\cite[Ch.~8]{BCSS}.
Smale showed that there exists $\alpha_0 > 0$ such that if $\alpha < \alpha_0$,
then Newton iterations starting at $x_0$ will converge quadratically to a solution $x$ of the system $f$.
That is, the number of significant digits doubles with each Newton iteration.
%We note that the value of $\alpha_0$ computed by Smale was $0.130707$, but one
%can take $\alpha_0 = (13-3\sqrt{17})/4\approx 0.157671$~\cite[\S~8.3]{BCSS}.
With $\alpha$-theory, we may use numerical methods in place of symbolic methods in many
applications, e.g., counting the number of real solutions~\cite{HS12}, while retaining the certainty of symbolic methods.
While there has been some work studying the convergence of Newton iterations when the
system is overdetermined~\cite{DS00} (more equations than variables), certification for
solutions is only known to be possible for square systems.

Using a determinantal formulation, Schubert problems are prototypical 
overdetermined polynomial systems.  Our main result is 
Theorem~\ref{thm:main} which states there exists 
a natural reformulation of these systems as complete
intersections using bilinear equations, thereby enabling the certification of approximate solutions.

In the next section, we give the usual determinantal formulation of intersections of Schubert
varieties in local coordinates for the Grassmannian.
In Section ~\ref{dual}, we reformulate Schubert problems as complete intersections by solving
a dual problem in a larger space, exchanging high-degree determinantal equations for bilinear
equations.
Finally, in Section ~\ref{future}, we sketch a hybrid aproach and 
discuss generalizations of our formulation.  

%%%%%%%%%%%%%%%%%%%%%%%%%%%%%%%%%%%%%%%%%%%%%%%%%%%%%%%%%%%%%%%%%%%%%%%%%%%%%%%%%%%%%%%%%%%%%%
\section{Schubert Calculus}\label{SCalc}

The solutions to a problem in the Schubert calculus are the points of an intersection of
Schubert varieties in a Grassmannian.
These intersections are formulated as systems of polynomial equations in 
local coordinates for the Grassmannian, which we now present.

Fix positive integers $k<n$ and let \defcolor{$V$} be a 
complex vector space of dimension $n$.
The set of all $k$-dimensional linear subspaces of $V$, denoted \defcolor{$\Gr(k,V)$}, is the
\demph{Grassmannian of $k$-planes in $V$}.

Let \defcolor{$\binom{[n]}{k}$} denote the set of sublists of $[n]:=(1,2,\dotsc,n)$ of
cardinality $k$. 
A \demph{Schubert (sub)variety} $X_\beta\Fdot\subset\Gr(k,V)$ is given by the data of a \demph{Schubert condition}
$\defcolor{\beta}\in\binom{[n]}{k}$ and a \demph{(complete) flag}
$\defcolor{\Fdot}:F_1\subset F_2\subset \dotsb \subset F_n =V$ of linear subspaces with
$\dim F_i=i$ where
 \begin{equation}\label{Eq:SchubertVariety}
   \defcolor{X_\beta\Fdot}\ :=\ 
   \{H\in \Gr(k,V)\ \mid\ 
     \dim(H\cap F_{\beta_i})\geq i,\mbox{ for  }i=1,\dotsc,k\}\,.
 \end{equation}
That is, the Schubert variety $X_\beta\Fdot$ 
is the set of $k$-planes satisfying the Schubert condition $\beta$
with respect to the flag $\Fdot$.

There are two standard formulations for a Schubert variety $X_\beta\Fdot$,
one as an implicit subset of $\Gr(k,V)$ given by a system of equations, and the
other explicity, as a parametrized subset of $\Gr(k,V)$.
They both begin with local coordinates for the Grassmannian.
An ordered basis $\be_1,\dotsc,\be_n$ for $V$ yields an identification of $V$ with $\C^n$
and leads to a system of local coordinates for $\Gr(k,V)$ given by matrices
$X\in\C^{k\times(n-k)}$. 
For this, the $k$-plane associated to a matrix $X$ is the row space of the matrix
$[X:I_k]$ where $I_k$ is the $k\times k$ identity matrix.
If $X=(x_{i,j})_{i=1,\dotsc,k}^{j=1,\dotsc,n-k}$, then this row space
is the span of the vectors $h_i:=\sum_{j=1}^{n-k}\be_{j}x_{i,j} + \be_{n-k+i}$ for
$i=1,\dotsc,k$.

Observe that a flag $\Fdot$ may be given by an ordered basis
\defcolor{$\bff_1,\dotsc,\bff_n$} for $V$, where $F_\ell$ is the linear span of
$\bff_1,\dotsc,\bff_\ell$.
Writing this basis $\{\bff_i\}$ in terms of the basis $\{\be_j\}$ gives a matrix
which we also write as $\Fdot$.  The space $F_\ell$ is the linear span of first $\ell$
rows of the matrix $\Fdot$.  The submatrix of $\Fdot$ consisting of the first $\ell$
rows will also be written as $F_\ell$. 

In the local coordinates $[X:I_k]$ for $\Gr(k,V)$, the Schubert variety $X_\beta\Fdot$ is
defined by 
 \begin{equation}\label{Eq:Rank}
   \rank \left[\begin{matrix}X\,:\,I_k\\ F_{\beta_i}\end{matrix}\right]
    \ \leq\ \beta_i+k-i
    \mbox{ \ for }i=1,\dotsc,k\,.
 \end{equation}
These rank conditions are equivalent to the vanishing of determinantal equations since the
condition $\rank(M)\leq a{-}1$ is equivalent to the vanishing of all $a\times a$ minors
(determinants of $a\times a$ submatrices) of $M$.
These determinants are polynomials in the entries of $X$ of degree up to
$\min\{k,n{-}k\}$, and there are
 \[
   \sum_{i=1}^k \tbinom{n}{\beta_i+k-i+1}\tbinom{k+\beta_i}{\beta_i+k-i+1}
 \]
of them.
If $\beta=(n{-}k\,,n{-}k{+}2\,,n{-}k{+}3\,,\dotsc,n)$, we write
%
%  Do not scale simple images.  Create them at the intended resolution.
%
$\beta=\I$ and since the determinant is the only minor required to vanish, $X_{\sI}\Fdot$ is
a hypersurface in $\Gr(k,V)$.
In all other cases, while there are linear dependencies among the minors, any maximal
linearly independent subset $S$ of minors remains overdetermined,  
i.e., $\#(S)>\codim_{\Gr(k,V)}X_\beta \Fdot$.

For the second formulation, consider the coordinate flag $\Edot$ associated to the
ordered basis $\be_1,\dotsc,\be_n$, so that $E_\ell$ is spanned by $\be_1,\dotsc,\be_\ell$.
The Schubert variety $X_\beta\Edot$ has a system of local coordinates similar to those
for $\Gr(k,V)$.
Consider the set of $k\times n$ matrices $\defcolor{M_\beta}=(m_{i,j})$ whose entries
satisfy
\[
   m_{i,\beta_j}\ =\ \delta_{i,j}\,,\qquad
   m_{i,j}\ =\ 0\mbox{ if }j>\beta_i\,,
\]
and the remaining entries are unconstrained.
These unconstrained entries identify $M_\beta$ with $\C^{\sum_i (\beta_i-i)}$.
The association of a matrix in $M_\beta$ to its row space yields
a parametrization of an open subset of the Schubert 
variety $X_\beta\Edot$ that defines local coordinates.

%%%%%%%%%%%%%%%%%%%%%%%%%%%%%%%%%%%%%%%%%%%%%%%%%%%%%%%%%%%%%%%%%%%%%%%%%%%%%%%
\begin{example}
When $k=3$ and $n=7$ and $\beta=(2,5,7)$ we have
\[
  M_{257}\  =\ 
   \left[\begin{matrix}
    m_{11} & 1 & 0      & 0      & 0 & 0      & 0 \\
    m_{21} & 0 & m_{23} & m_{24} & 1 & 0      & 0 \\
    m_{31} & 0 & m_{33} & m_{34} & 0 & m_{36} & 1 \\
   \end{matrix}\right]\,.
\]
\end{example}
%%%%%%%%%%%%%%%%%%%%%%%%%%%%%%%%%%%%%%%%%%%%%%%%%%%%%%%%%%%%%%%%%%%%%%%%%%%%%%%

%%%%%%%%%%%%%%%%%%%%%%%%%%%%%%%%%%%%%%%%%%%%%%%%%%%%%%%%%%%%%%%%%%%%%%%%%%%%%%%
\begin{lemma}\label{oneCondition}
 The association of a matrix in $M_\beta$ to its row space identifies $M_\beta$ with a
 dense open subset of $X_\beta\Edot$.
 If $\Fdot$ is a complete flag given by a $n\times n$ matrix $\Fdot$, then the association
 of a matrix $H$ in $M_\beta$ to the row space of the product $H\Fdot$ identifies
 $M_\beta$ with a dense open subset of  $X_\beta\Fdot$.
\end{lemma}
%%%%%%%%%%%%%%%%%%%%%%%%%%%%%%%%%%%%%%%%%%%%%%%%%%%%%%%%%%%%%%%%%%%%%%%%%%%%%%%

%%%%%%%%%%%%%%%%%%%%%%%%%%%%%%%%%%%%%%%%%%%%%%%%%%%%%%%%%%%%%%%%%%%%%%%%%%%%%%%
\begin{proof}
 The first statement is the assertion that $M_\beta$ gives local coordinates for
 $X_\beta\Edot$, which is classical~\cite[p.~147]{Fulton}.
 The second statement follows from the observation that if $g\in GL(n,\C)$ is an invertible
 linear transformation, a $k$-plane $H$ lies in $X_\beta\Edot$ if and only if
 $Hg$  lies in $(X_\beta\Edot)g=X_\beta(\Edot g)$.
 The lemma follows as the transformation $g$ with $\Fdot=\Edot g$ is given by
 the matrix $\Fdot$. 
\end{proof}
%%%%%%%%%%%%%%%%%%%%%%%%%%%%%%%%%%%%%%%%%%%%%%%%%%%%%%%%%%%%%%%%%%%%%%%%%%%%%%%

Counting parameters gives a formula for the codimension of $X_\beta\Fdot$ in $\Gr(k,V)$ namely
\[
   \defcolor{|\beta|}\ :=\ \codim X_\beta \Fdot\ =\ k(n - k) - \sum_i (\beta_i-i)\,.
\]

For $\beta,\gamma\in\tbinom{[n]}{k}$, there is a smaller system of local coordinates
$M_\beta^\gamma$ explicitly parametrizing an intersection of two Schubert varieties. 
Let $\Edot'$ be the coordinate flag opposite to $\Edot$ in which
$E_\ell':=\langle\be_n,\dotsc,\be_{n+1-\ell}\rangle$.
As the flags $\Edot, \Edot'$ are in linear general position, 
the definition of a Schubert variety~\eqref{Eq:SchubertVariety} implies that 
the intersection $X_\beta\Edot \cap X_\gamma\Edot'$ is
nonempty if and only if we have $n+1-\gamma_{k+1-i}\leq \beta_i$ for each
$i=1,\dotsc,k$.
When this holds, the intersection has a system of local coordinates
given by the row space of $k\times n$ matrices $\defcolor{M_\beta^\gamma}=(m_{i,j})$
in which
\[
    m_{i,j}\ :=\ 0 \ \mbox{ if }\ j\not\in[n+1-\gamma_{k+1-i},\beta_i]
    \ \mbox{ and }\ m_{i,\beta_i}\ :=\ 1\,,\ \mbox{ for }\ i=1,\dotsc,k\,.
\]
The unconstrained entries of $M_\beta^\gamma$ identify it with the affine space
$\C^{k(n-k)-|\beta|-|\gamma|}$.

%%%%%%%%%%%%%%%%%%%%%%%%%%%%%%%%%%%%%%%%%%%%%%%%%%%%%%%%%%%%%%%%%%%%%%%%%%%%%%%
\begin{example}
When $k=3$, $n=7$, $\beta=(2,5,7)$, and $\gamma=(3,5,7)$ we have
\[
  M_{257}^{357}\  =\ 
   \left[\begin{matrix}
    m_{11} & 1 & 0      & 0      & 0      & 0      & 0 \\
    0      & 0 & m_{23} & m_{24} & 1      & 0      & 0 \\
    0      & 0 & 0      & 0      & m_{35} & m_{36} & 1 \\
   \end{matrix}\right]\,.
\]
\end{example}
%%%%%%%%%%%%%%%%%%%%%%%%%%%%%%%%%%%%%%%%%%%%%%%%%%%%%%%%%%%%%%%%%%%%%%%%%%%%%%%

%%%%%%%%%%%%%%%%%%%%%%%%%%%%%%%%%%%%%%%%%%%%%%%%%%%%%%%%%%%%%%%%%%%%%%%%%%%%%%%
\begin{lemma}\label{twoConditions}
 The association of a matrix in $M_\beta^\gamma$ to its row space identifies
 $M_\beta^\gamma$ with a dense open subset of $X_\beta\Edot\cap X_\gamma\Edot'$.
 Suppose that $\Fdot$ and $\Fdot'$ are flags in general position in $V$ and that $g$ is an
 invertible linear transformation such that $\Fdot=\Edot g$ and $\Fdot'=\Edot'g$.
 Then the set of matrices $M_\beta^\gamma g$ parametrizes a dense open subset of 
 $X_\beta\Fdot\cap X_\gamma\Fdot'$.
\end{lemma}
%%%%%%%%%%%%%%%%%%%%%%%%%%%%%%%%%%%%%%%%%%%%%%%%%%%%%%%%%%%%%%%%%%%%%%%%%%%%%%%

%%%%%%%%%%%%%%%%%%%%%%%%%%%%%%%%%%%%%%%%%%%%%%%%%%%%%%%%%%%%%%%%%%%%%%%%%%%%%%%
As with Lemma~\ref{oneCondition}, this is classical.
The existence of the linear transformation $g$ sending the coordinate flags 
$\Edot,\Edot'$ to the flags $\Fdot$ and $\Fdot'$ is an exercise in linear algebra.
We often assume that two of our flags are the coordinate flags $\Edot,\Edot'$.

A \demph{Schubert problem} on $\Gr(k,V)$ is a list of Schubert conditions
$\defcolor{\bbeta}=(\beta^1,\dotsc,\beta^\ell)$ with
$\sum_{i=1}^\ell |\beta^i|=k(n{-}k)$.
Given a Schubert problem $\bbeta$ and a list $\Fdot^1,\dotsc,\Fdot^\ell$, the
intersection
 \begin{equation}\label{Eq:instance}
    X_{\beta^1}\Fdot^1 \cap \cdots \cap X_{\beta^\ell}\Fdot^\ell
 \end{equation}
is an \demph{instance} of the Schubert problem $\bbeta$.
When the flags are general, the intersection~\eqref{Eq:instance} is
transverse~\cite{KL74}.
The points in the intersection are the \demph{solutions} to this instance of the Schubert
problem, and their number \defcolor{$N(\bbeta)$} may be calculated using algorithms based
on the Littlewood-Richardson rule.

%%%%%%%%%%%%%%%%%%%%%%%%%%%%%%%%%%%%%%%%%%%%%%%%%%%%%%%%%%%%%%%%%%%%%%%%%%%%%%%%%%%%
\begin{example}\label{Ex:SP}
  Suppose that $k=2$, $n=6$, and $\bbeta=(\beta,\beta,\beta,\beta)$ where $\beta=(3,6)$.
  Since $|(3,6)|=2$ and $2+2+2+2=2(6-2)=\dim\Gr(2,\C^6)$, $\bbeta$ is a Schubert problem
  on $\Gr(2,\C^6)$.
  One can verify that $N(\bbeta)=3$.
\end{example}
%%%%%%%%%%%%%%%%%%%%%%%%%%%%%%%%%%%%%%%%%%%%%%%%%%%%%%%%%%%%%%%%%%%%%%%%%%%%%%%%%%%%

We wish to solve instances~\eqref{Eq:instance} of a Schubert problem $\bbeta$
formulated as a system of equations given by the rank conditions~\eqref{Eq:Rank}.
Rather than use the local coordinates $[X:I_k]$ for the Grassmannian, which has $k(n{-}k)$
variables, we may use $M_{\beta^1}$ as local coordinates for $X_{\beta^1}\Fdot^1$, which gives
$k(n{-}k)-|\beta^1|$ variables.
When $\Fdot^1$ and $\Fdot^2$ are in linear general position we may use
$M_{\beta^1}^{\beta^2}$ as local coordinates for
$X_{\beta^1}\Fdot^1\cap X_{\beta^2}\Fdot^2$, which uses only
$k(n{-}k)-|\beta^1|-|\beta^2|$ variables.
These smaller sets of local coordinates often lead to more efficient computation.

%%%%%%%%%%%%%%%%%%%%%%%%%%%%%%%%%%%%%%%%%%%%%%%%%%%%%%%%%%%%%%%%%%%%%%%%%%%%%%%%%%%%
\begin{example}
 For the Schubert problem of Example~\ref{Ex:SP}, if we assume that $\Fdot^1=\Edot$
 and $\Fdot^2=\Edot'$, then we may use the local coordinates $M_{36}^{36}$.
 In these local coordinates, the essential rank conditions~\eqref{Eq:Rank} on the
 Schubert variety $X_{36}\Fdot$ are equivalent to the vanishing of all full-sized
 ($5\times 5$) minors of the $5\times 6$ matrix whose first two rows are $M_{36}^{36}$
 and last three are $F_3$.  In particular, we have $2\cdot 6=12$ equations of 
 degree at most 2 in four variables, which have three common solutions.
 The maximal linearly independent set of equations
 consists of six equations with four variables, which remains overdetermined.
\end{example}
%%%%%%%%%%%%%%%%%%%%%%%%%%%%%%%%%%%%%%%%%%%%%%%%%%%%%%%%%%%%%%%%%%%%%%%%%%%%%%%%%%%%

%%%%%%%%%%%%%%%%%%%%%%%%%%%%%%%%%%%%%%%%%%%%%%%%%%%%%%%%%%%%%%%%%%%%%%%%%%%%%%%%%%%%%%%
% the Methods of Computation section does not belong here.
% it may not belong anywhere.
% I've moved it past the \end{document} for now.
%%%%%%%%%%%%%%%%%%%%%%%%%%%%%%%%%%%%%%%%%%%%%%%%%%%%%%%%%%%%%%%%%%%%%%%%%%%%%%%%%%%%%%%
\section{Primal-Dual Formulation of Schubert Problems}\label{dual}

Large computational experiments~\cite{secant,monotone_MEGA,So00b} have used
symbolic computation to solve billions of instances of Schubert problems, producing
compelling conjectures, some of which have since been
proved~\cite{EG02,Mod4,MTV_Annals,MTV_R}.
These experiments required certified symbolic methods in characteristic zero and were
constrained by the limits of computability imposed by the complexity of Gr\"obner basis
computation.
Roughly, Schubert problems with more than 100 solutions or whose formulation involves more
than 16 variables are infeasible, and a typical problem at the limit of feasibility has 30
solutions in 9~variables.

We are not alone in the belief that numerical methods offer the best route for studying
larger Schubert problems. 
This led to the development of specialized numerical algorithms for Schubert
problems, such as the Pieri homotopy algorithm~\cite{HSS98}, which was used to study a
problem with 17589 solutions~\cite{LS}.
It is also driving the development~\cite{svv} and implementation~\cite{LMSVV} of the
Littlewood-Richardson homotopy, based on Vakil's geometric Littlewood-Richardson
rule \cite{Va06a,Va06b}.  Regeneration~\cite{HSW10} offers 
another numerical approach for Schubert problems.

As explained in Section~\ref{SCalc}, traditional formulations of Schubert problems 
typically lead to overdetermined systems of 
polynomials of degree $\min\{k,n{-}k\}$, expressed in whichever of the
systems $[X\,:\,I_k]$, $M_{\beta}$, or $M_\beta^\gamma$ of local coordinates is relevant.
We present an alternative formulation of Schubert varieties and Schubert problems as
complete intersections of bilinear equations involving more variables.

Recall that $V$ is a vector space equipped with a basis
$\be_1,\dotsc,\be_n$.
Let \defcolor{$V^*$} be its dual vector space and
\defcolor{$\be_1^*,\dotsc,\be_n^*$} be the corresponding dual basis.
For every $k=1,\dotsc,n{-}1$, the association of a $k$-plane $H\subset V$ to its
annihilator $H^\perp\subset V^*$ is the cannonical isomorphism, written
$\defcolor{\perp}$, between the Grassmannian $\Gr(k,V)$ and its \demph{dual Grassmannian}
\defcolor{$\Gr(n{-}k,V^*)$}.
For a Schubert variety $X_\beta\Fdot\subset\Gr(k,V)$
we have $\defcolor{{\perp}(X_\beta\Fdot)} := \{H^\perp\mid H\in X_\beta\Fdot\}$,
which is a subset of $\Gr(n{-}k,n)$.
To identify ${\perp}(X_\beta\Fdot)$, we make some definitions.

Each flag $\Fdot$ on $V$ has a corresponding \demph{dual flag} \defcolor{$\Fdot^\perp$} on
$V^*$,
  \[
    \Fdot^\perp\ \colon\ 
    (F_{n-1})^\perp\subset (F_{n-2})^\perp \subset \cdots \subset (F_1)^\perp
     \subset V*\,,
  \]
which is a flag since $\dim(F_{i}) + \dim(F_{n-i})^\perp=n$.
For $\beta\in\binom{[n]}{k}$, a subset of $[n]$ of cardinality $k$, consider
$\defcolor{\beta^\perp}:=(j\ |\ n{+}1{-}j\in [n]\smallsetminus\beta)\in\tbinom{[n]}{n-k}$.
The map $\beta\mapsto\beta^\perp$ is a bijection.

%%%%%%%%%%%%%%%%%%%%%%%%%%%%%%%%%%%%%%%%%%%%%%%%%%%%%%%%%%%%%%%%%%%%%%%%%%%%%%%
\begin{lemma}\label{dualSchub}
 For a Schubert variety $X_\beta\Fdot\subset\Gr(k,V)$, we have
 ${\perp}(X_\beta\Fdot) = X_{\beta^\perp}\Fdot^\perp$.
\end{lemma}
%%%%%%%%%%%%%%%%%%%%%%%%%%%%%%%%%%%%%%%%%%%%%%%%%%%%%%%%%%%%%%%%%%%%%%%%%%%%%%%

Note that $X_\beta\Fdot= {\perp}(X_{\beta^\perp}\Fdot^\perp)$.
We call $X_\beta\Fdot$ and $X_{\beta^\perp}\Fdot^\perp$ \demph{dual Schubert varieties}.
%%%%%%%%%%%%%%%%%%%%%%%%%%%%%%%%%%%%%%%%%%%%%%%%%%%%%%%%%%%%%%%%%%%%%%%%%%%%%%%
\begin{proof}
 Observe that if $\Fdot$ is a flag and $H$ a linear subspace, then $\dim H\cap F_{b}\geq
 a$ implies that $\dim H\cap F_{b+1}\geq a$.
 Thus the definition~\eqref{Eq:SchubertVariety} of Schubert variety is equivalent to
 \begin{equation}\label{Eq:SchubVarAlt}
   \defcolor{X_\beta\Fdot}\ :=\ 
   \{H\in \Gr(k,V)\ \mid\ 
     \dim(H\cap F_{i})\geq \#\{\beta \cap [i]\},\mbox{ for  }i=1,\dotsc,n\}\,.
 \end{equation}
 For every $H\in\Gr(k,V)$ and all $i=1,\dotsc,n$, the following are equivalent:
 \begin{multline*}
   \dim H\cap F_i\ \geq\ \#(\beta \cap [i])\\
  \Leftrightarrow\ 
   \makebox[330pt][l]{$\dim(\Span\{ H,F_i\})\ \leq\ k+i-\#(\beta \cap [i])\ =\ 
    i + \#(\beta \cap \{i+1,\dotsc,n\})$}\\
  \Leftrightarrow\ 
    \makebox[330pt][l]{$\dim(\Span\{ H,F_i\}^\perp)\ \geq
      n - i - \#(\beta \cap \{i+1,\dotsc,n\})$} \\
  \Leftrightarrow\ 
    \makebox[372pt][l]{$\dim(H^\perp \cap F^\perp_{n-i})\ \geq
      n - i - \#(\beta \cap \{i+1,\dotsc,n\})\,.$}
 \end{multline*}
 Since $n-i-\#(\beta \cap \{i{+}1,\dotsc,n\})=\#(\beta^\perp\cap[n{-}i])$, the
 lemma follows from \eqref{Eq:SchubertVariety}.
\end{proof}
%%%%%%%%%%%%%%%%%%%%%%%%%%%%%%%%%%%%%%%%%%%%%%%%%%%%%%%%%%%%%%%%%%%%%%%%%%%%%%%

Let $\Delta\colon\Gr(k,V)\rightarrow\Gr(k,V)\times\Gr(n{-}k,V^*)$ be the 
graph of the canonical isomorphism $\perp\colon\Gr(v,V)\to\Gr(n{-}k,V^*)$.
We call $\Delta$ the \demph{dual diagonal} map.
In this context, the classical reduction to the diagonal becomes the following.

%%%%%%%%%%%%%%%%%%%%%%%%%%%%%%%%%%%%%%%%%%%%%%%%%%%%%%%%%%%%%%%%%%%%%%%%%%%%%%%
\begin{lemma}\label{L:diagonal}
 Let $A,B\subset\Gr(k,V)$.  
 Then $\Delta(A\cap B)= (A\times{\perp}(B))\cap\Delta(\Gr(k,V))$.
\end{lemma}
%%%%%%%%%%%%%%%%%%%%%%%%%%%%%%%%%%%%%%%%%%%%%%%%%%%%%%%%%%%%%%%%%%%%%%%%%%%%%%%

We call the reduction to the diagonal of Lemma~\ref{L:diagonal} the
\demph{primal-dual} reformulation of the intersection $A\cap B$.
We use this primal-dual reformulation to express Schubert problems as
complete intersections given by bilinear equations.
For this, suppose that $M$ is a $k\times n$ matrix whose row space $H$ is a $k$-plane in
$V$ and $N$ is a $n\times(n{-}k)$ matrix whose column space $K$ is a $(n{-}k)$-plane in
$V^*$.
(The coordinates of the matrices---columns for $M$ and rows for $N$---are with respect to 
the bases $\be_i$ and $\be^*_j$.)
Then $H^\perp=K$ if and only if $MN=0_{k\times(n{-}k)}$, giving $k(n{-}k)$ bilinear
equations in the entries of $M$ and $N$ for $\Delta(\Gr(k,V))$.
We deduce the fundamental lemma underlying our reformulation.

%%%%%%%%%%%%%%%%%%%%%%%%%%%%%%%%%%%%%%%%%%%%%%%%%%%%%%%%%%%%%%%%%%%%%%%%%%%%%%%%%%%%%%%%
\begin{lemma}\label{L:bilinear}
 Let $A,B$ be two subsets of $\Gr(k,V)$ and suppose that $M$ is a set of 
 $k\times n$ matrices parametrizing $A$ (via row space) and that $N$ is a set of   
 $n\times(n{-}k)$ matrices parametrizing ${\perp}(B)$ (via column space).
 Then $\Delta(A\cap B)$ is the subset of $A\times{\perp}(B)$ defined in its parametrization
 $M\times N$ by the equations $MN=0_{k\times(n{-}k)}$.
\end{lemma}
%%%%%%%%%%%%%%%%%%%%%%%%%%%%%%%%%%%%%%%%%%%%%%%%%%%%%%%%%%%%%%%%%%%%%%%%%%%%%%%%%%%%%%%%

%%%%%%%%%%%%%%%%%%%%%%%%%%%%%%%%%%%%%%%%%%%%%%%%%%%%%%%%%%%%%%%%%%%%%%%%%%%%%%%
\begin{example}\label{Ex:bilinear}
  We explore some consequences of Lemma~\ref{L:bilinear}.
  Suppose that $A\subset\Gr(k,V)$ is the subset parametrized by matrices $[X:I_k]$ for $X$
  a $k\times(n{-}k)$ matrix.
  Then ${\perp}(A)\subset\Gr(n{-}k,V^*)$ is parametrized by
  matrices $[I_{n-k}:Y^T]^T$, where $Y$ is a $k\times(n{-}k)$ matrix.
  The bilinear equations defining $\Delta(A)\subset A\times{\perp}(A)$ coming from
  these parametrizations are $X+Y=0$.
  Thus if $H$ is the row space of $[I_k:X]$, then $H^\perp$ is the column space of
  $[I_{n-k}:-X^T]^T$.

  Let $\beta\in\binom{[n]}{k}$ and $\Fdot$ be a flag in $V$, and 
  suppose that $N_\beta\simeq\C^{k(n-k)-|\beta|}$ is a set of $n\times(n-k)$ matrices
  parametrizing $X_{\beta^\perp}\Fdot^\perp$ as in Lemma~\ref{oneCondition}.
  Let $X^\circ_\beta\Fdot$ be the open subset of $X_\beta\Fdot$ such that
  ${\perp}(X^\circ_\beta\Fdot)$ is the subset parametrized by $N_\beta$.
  Given a set $M$ of $k\times n$ matrices which parametrize an open subset $\calO$ of
  $\Gr(k,V)$, Lemma~\ref{L:bilinear} implies that the bilinear equations $MN_\beta=0$
  in $M\times N_\beta$ define $\Delta(\calO\cap X^\circ_\beta\Fdot)$ as a subset of 
  $\calO\times{\perp}(X^\circ_\beta\Fdot)$.

  When $\calO\cap X_\beta\Fdot\neq\emptyset$, we call this pair of parametrizations $M$
  for $\Gr(k,V)$ and $N_\beta$ for $X_{\beta^\perp}\Fdot^\perp$, together with the biliear
  equations $MN_\beta=0$, the \demph{primal-dual formulation} of the Schubert variety 
  $X_\beta\Fdot$.   
  It is $k(n-k)$ equations in $k(n{-}k)+k(n{-}k)-|\beta|$ variables (at least when $M$ is
  identified with affine space of dimension $k(n{-}k)$.)
  Thus we have identified $\Delta(X_\beta\Fdot)$ as a complete intersection in a system of
  local coordinates. 
\end{example}
%%%%%%%%%%%%%%%%%%%%%%%%%%%%%%%%%%%%%%%%%%%%%%%%%%%%%%%%%%%%%%%%%%%%%%%%%%%%%%%

We extend this primal-dual formulation of a Schubert variety to a
formulation of a Schubert problem as a complete intersection of bilinear equations.
This uses a dual diagonal map $\Delta$ to the small diagonal in a larger product of
Grassmannians. 
Define
 \[
   \Delta^\ell\ \colon\ 
   \Gr(k,V)\rightarrow\Gr(k,V)\times \bigl(\Gr(n-k,V^*)\bigr)^{\ell-1}\,,
 \]
by sending $H\mapsto(H,H^\perp,\dotsc,H^\perp)$.
Classical reduction to the diagonal extends to multiple factors, giving the following.

%%%%%%%%%%%%%%%%%%%%%%%%%%%%%%%%%%%%%%%%%%%%%%%%%%%%%%%%%%%%%%%%%%%%%%%%%%%%%%%
\begin{lemma}\label{L:ellDiagonal}
 Let $A_1,\dotsc,A_{\ell}\subset\Gr(k,V)$.
 Then 
\[
   \Delta^\ell(A_1\cap \cdots \cap A_\ell)\ =\ 
   (A_1\times {\perp}(A_2)\times \cdots \times {\perp}(A_\ell))
    \bigcap \Delta^\ell(\Gr(k,V))\,.
\]
\end{lemma}
%%%%%%%%%%%%%%%%%%%%%%%%%%%%%%%%%%%%%%%%%%%%%%%%%%%%%%%%%%%%%%%%%%%%%%%%%%%%%%%

Lemma ~\ref{L:bilinear} extends to the dual diagonal of many factors.

%%%%%%%%%%%%%%%%%%%%%%%%%%%%%%%%%%%%%%%%%%%%%%%%%%%%%%%%%%%%%%%%%%%%%%%%%%%%%%%
\begin{lemma}\label{L:ellBilinear}
 Let $A_1,\dotsc,A_{\ell}\subset\Gr(k,V)$, and suppose that $M$ is a set of 
 $k\times n$ matrices parametrizing $A_1$ and that $N_i$ is a $n\times (n{-}k)$ matrix
 parametrizing ${\perp}(A_i)$ for $i=2,\dotsc,\ell$.
 Then $\Delta^\ell(A_1\cap \cdots \cap A_\ell)$ is the subset of 
 $A_1\times {\perp}(A_2) \times \cdots \times {\perp}(A_\ell)$ defined in the 
 parametrization $M\times N_2 \times \cdots \times N_\ell$ by the equations
 $MN_i=0_{k\times(n{-}k)}$ for $i=2,\dotsc,\ell$. 
\end{lemma}
%%%%%%%%%%%%%%%%%%%%%%%%%%%%%%%%%%%%%%%%%%%%%%%%%%%%%%%%%%%%%%%%%%%%%%%%%%%%%%%

%%%%%%%%%%%%%%%%%%%%%%%%%%%%%%%%%%%%%%%%%%%%%%%%%%%%%%%%%%%%%%%%%%%%%%%%%%%%%%%
\begin{theorem}\label{thm:main}
 Any sufficiently general instance of a Schubert problem $\bbeta$ may be 
 reformulated as a complete intersection of bilinear equations in the coordinates
 $(M_{\beta^1},M_{\beta^{2{\perp}}},\dotsc,M_{\beta^{\ell{\perp}}})$.
\end{theorem}
%%%%%%%%%%%%%%%%%%%%%%%%%%%%%%%%%%%%%%%%%%%%%%%%%%%%%%%%%%%%%%%%%%%%%%%%%%%%%%%

%%%%%%%%%%%%%%%%%%%%%%%%%%%%%%%%%%%%%%%%%%%%%%%%%%%%%%%%%%%%%%%%%%%%%%%%%%%%%%%
\begin{proof}
 A sufficiently general instance of $\bbeta$ is zero-dimensional with $N(\bbeta)$
 solutions. 
 The result follows from Lemma~\ref{L:ellBilinear} in which $A_1$ is the open subset of
 $X_{\beta^1}\Fdot^1$ parametrized by $M_{\beta^1}$ and for $i>1$, $A_i$ is the open
 subset of $X_{\beta^i}\Fdot^i$ with ${\perp}A_i$ parametrized by
 $M_{\beta^{\ell{\perp}}}$.
 This gives $k(n{-}k)(\ell{-}1)$ bilinear equations in $k(n{-}k)(\ell{-}1)$ variables.
\end{proof}
%%%%%%%%%%%%%%%%%%%%%%%%%%%%%%%%%%%%%%%%%%%%%%%%%%%%%%%%%%%%%%%%%%%%%%%%%%%%%%%

Theorem~\ref{thm:main} provides a formulation of an instance of a Schubert problem 
as a square system, to which the certification afforded by Smale's $\alpha$-theory may be
applied. 
This rectifies the fundamental obstruction to using numerical methods 
in place of certified symbolic methods for solving Schubert problems.

We apply Lemma~\ref{L:bilinear} to the coordinates of Lemma~\ref{twoConditions}.

%%%%%%%%%%%%%%%%%%%%%%%%%%%%%%%%%%%%%%%%%%%%%%%%%%%%%%%%%%%%%%%%%%%%%%%%%%%%%%%
\begin{example}

  Given indices $\beta^1,\dotsc,\beta^4\in\binom{[n]}{k}$ and flags
  $\Fdot^1,\dotsc,\Fdot^4$ in $V$.
  If $M_{\beta^1}^{\beta^2}$
  parametrizes an open subset of $X_{\beta^1}\Fdot^1\cap X_{\beta^2}\Fdot^2$ and 
  $N_{\beta^3}^{\beta^4}$ parametrizes an open subset of 
  ${\perp}(X_{\beta^3}\Fdot^3\cap X_{\beta^4}\Fdot^4)$ as in Lemma~\ref{twoConditions}.
  Then the bilinear equations $M_{\beta^1}^{\beta^2} N_{\beta^3}^{\beta^4}=0$ define
  the intersection
\[
    \Delta\bigl(X_{\beta^1}\Fdot^1\cap X_{\beta^2}\Fdot^2
     \cap X_{\beta^3}\Fdot^3\cap X_{\beta^4}\Fdot^4\bigr)
\]
  as a subset of $(X_{\beta^1}\Fdot^1\cap X_{\beta^2}\Fdot^2) \times {\perp}(X_{\beta^3}\Fdot^3\cap X_{\beta^4}\Fdot^4)$. 
\end{example}
%%%%%%%%%%%%%%%%%%%%%%%%%%%%%%%%%%%%%%%%%%%%%%%%%%%%%%%%%%%%%%%%%%%%%%%%%%%%%%%

%%%%%%%%%%%%%%%%%%%%%%%%%%%%%%%%%%%%%%%%%%%%%%%%%%%%%%%%%%%%%%%%%%%%%%%%%%%%%%%

This suggests an improvement of the efficiency of the primal-dual formulation of Schubert
problems. 

%%%%%%%%%%%%%%%%%%%%%%%%%%%%%%%%%%%%%%%%%%%%%%%%%%%%%%%%%%%%%%%%%%%%%%%%%%%%%%%
\begin{corollary}\label{C:primal-dual}
 Any sufficiently general instance of a Schubert problem $\bbeta$ given by the
 intersection of $\ell$ Schubert varieties may be naturally reformulated as a
 complete intersection in $\lfloor \frac{\ell-1}{2} \rfloor k(n{-}k)$ variables.
\end{corollary}
%%%%%%%%%%%%%%%%%%%%%%%%%%%%%%%%%%%%%%%%%%%%%%%%%%%%%%%%%%%%%%%%%%%%%%%%%%%%%%%

%%%%%%%%%%%%%%%%%%%%%%%%%%%%%%%%%%%%%%%%%%%%%%%%%%%%%%%%%%%%%%%%%%%%%%%%%%%%%%%
\begin{proof}
 When $\ell$ is even one may reduce the number of equations and variables by
 parametrizing
 \[
   (X_{\beta^1}\Fdot^1 \cap X_{\beta^2} \Fdot^2) \times (X_{\beta^{3{\perp}}}
 \Fdot^{3{\perp}}\cap X_{\beta^{4{\perp}}}\Fdot^{4{\perp}}) \times\cdots \times 
 (X_{\beta^{\ell-1{\perp}}}\Fdot^{\ell-1{\perp}}\cap X_{\beta^{\ell{\perp}}}
 \Fdot^{\ell{\perp}})\,,
\]
 using local coordinates $(M_{\beta^1}^{\beta^2}, N_{\beta^3}^{\beta^4},\dotsc,
 N_{\beta^{\ell-1}}^{\beta^\ell})$.
 When $\ell$ is odd, the last factor is simply $X_{\beta_{\ell{\perp}}}\Fdot^{\ell{\perp}}$, and
 the local coordinates are $(M_{\beta^1}^{\beta^2}, N_{\beta^3}^{\beta^4},
 \dotsc,N_{\beta^{\ell-2}}^{\beta^\ell-1},N_{\beta^\ell})$.
\end{proof}
%%%%%%%%%%%%%%%%%%%%%%%%%%%%%%%%%%%%%%%%%%%%%%%%%%%%%%%%%%%%%%%%%%%%%%%%%%%%%%%

\section{Specialization and Generalization.}\label{future}

In the previous section we formulated a Schubert problem as a square system,
which enables the certification of output from numerical methods, but at the expense of
increasing the number of variables.
In many cases, it is possible to eliminate some variables without the system
becoming overdetermined.

Recall that $X_{\sI} \Fdot$ is a hypersurface
defined by one equation.
Given a Schubert problem $\bbeta=(\I,\dotsc,\I,\beta^m,\dotsc,\beta^\ell)$, we obtain a
square system using the primal formulation for the intersection of the first
$m{+}1$ Schubert varieties in local coordinates $M_{\beta^m}^{\beta^{m+1}}$.
While this generally introduces equations of higher degree, 
it reduces the number of variables.

%%%%%%%%%%%%%%%%%%%%%%%%%%%%%%%%%%%%%%%%%%%%%%%%%%%%%%%%%%%%%%%%%%%%%%%%%%%%%%%
\begin{example}
 We denote $\beta=(n{-}k{-}2\,,n{-}k{+}2\,,n{-}k{+}3\,,\dotsc,n)$ by $\beta=\III$.
 Consider the Schubert problem
 \[\bbeta = (\,\I\,,\I\,,\I\,,\I\,,\I\,,\I\,,\III\,,\III\,,\III\,,\III\,)\]
 in $\Gr(3,9)$.
 The primal formulation~\eqref{Eq:Rank} consists of 26 linearly independent determinants
 of degree at most 3 in $M_{\sIII}^{\sIII}$, which has dimension 12.
 Using the primal formulation for the intersection
 \[
   X_{\sI}\Fdot^1\cap\cdots\cap X_{\sI}\Fdot^6 \cap X_{\sIII}\Fdot^7 \cap X_{\sIII}\Fdot^8\,,
 \]
 and the dual formulation for the intersection
 \[
   X_{{\sIII}^{\perp}}\Fdot^{9{\perp}} \cap X_{{\sIII}^{\perp}}\Fdot^{10{\perp}}\,,
 \]
 this problem is reduced to a square system consisting of 18 bilinear equations 
 and 6 determinants in the 24 variables $(M_{\sIII}^{\sIII},M_{\sIII}^{\sIII})$.
 The full primal-dual formulation of Corolary~\ref{C:primal-dual} has 72 bilinear
 equations in 72 variables.
 The number of solutions is $N(\bbeta)=437$.
\end{example}
%%%%%%%%%%%%%%%%%%%%%%%%%%%%%%%%%%%%%%%%%%%%%%%%%%%%%%%%%%%%%%%%%%%%%%%%%%%%%%%

This primal-dual formulation and its improvements of Corolary~\ref{C:primal-dual} and that
given above may be used either to solve an instance of a Schubert problem on a
Grassmannian or to to certify solutions to an instance of a Schubert problem computed by
one of the other methods mentioned in the Introduction. 

This primal-dual formulation extends with little change (Lemma~\ref{twoConditions} and its
consequences do not always apply) to Schubert problems on all classical flag
varieties---those of types $A$, $B$, $C$, and $D$.  
The fundamental reason is that Schubert varieties on classical flag varieties all have
parametrizations in terms of local coordinates as in Lemma~\ref{oneCondition}, and
dual diagonal maps generalizing $\Delta$ and $\Delta^\ell$.

We have implemented these techniques in Schubert problems on Grassmannians, and
will implement this for other flag varieties.

\providecommand{\bysame}{\leavevmode\hbox to3em{\hrulefill}\thinspace}
\providecommand{\MR}{\relax\ifhmode\unskip\space\fi MR }
\providecommand{\MRhref}[2]{%
  \href{http://www.ams.org/mathscinet-getitem?mr=#1}{#2}
}
\providecommand{\href}[2]{#2}

\end{document}